\documentclass[11pt,reqno,a4paper]{amsart}

\usepackage{amsfonts, amsmath, amsthm, amssymb}
\usepackage[english]{babel}
\usepackage[mathcal]{eucal}
\usepackage{geometry}
\usepackage{relsize}
\usepackage{xcolor}
\usepackage{longtable}
\usepackage{caption}
\usepackage{setspace}
\title[]{On asymptotic normality in estimation after a group sequential trial}
\author{Ben Berckmoes, Anna Ivanova, Geert Molenberghs}
\thanks{Ben Berckmoes is post doctoral fellow at the Fund for Scientific Research of Flanders (FWO)}
\thanks{Financial support from the IAP research network \#P7/06 of the Belgian Government (Belgian Science Policy) is gratefully acknowledged.}

\date{}

\begin{document}

\maketitle

\newtheorem{pro}{Proposition}
\newtheorem{lem}[pro]{Lemma}
\newtheorem{thm}[pro]{Theorem}
\newtheorem{de}[pro]{Definition}
\newtheorem{co}[pro]{Comment}
\newtheorem{no}[pro]{Notation}
\newtheorem{vb}[pro]{Example}
\newtheorem{vbn}[pro]{Examples}
\newtheorem{gev}[pro]{Corollary}
\newtheorem{vrg}[pro]{Question}
\newtheorem{rem}[pro]{Remark}
\newtheorem{lemA}{Lemma}

\begin{abstract}
We prove that in many realistic cases, the ordinary sample mean after a group sequential trial is asymptotically normal if the maximal number of observations increases. We derive that it is often safe to use naive confidence intervals for the mean of the collected observations, based on the ordinary sample mean. Our theoretical findings are confirmed by a simulation study.
\end{abstract}

\section{Introduction}\label{sec:Intro}

Throughout this paper, we will tacitly assume that all random variables are defined on the same probability space. Furthermore, we will keep fixed 
\ \\
- parameters $\mu \in \mathbb{R}$ and $\sigma \in \mathbb{R}^+_0$,
\ \\
- a sequence  $X_1, X_2, \ldots$ of i.i.d. random variables with normal law $N(\mu,\sigma^2)$,
\ \\
- a natural number $L \in \mathbb{N}_0$, 
\ \\
- natural numbers $0 < k_1 < k_2 < \ldots < k_{L + 1}$, 
\ \\
- a Borel measurable map $\psi : \mathbb{R} \rightarrow \left[0,1\right]$,
\ \\
- a parameter $\gamma \in \mathbb{R}^+$.

Also, for each $n \in \mathbb{N}_0$, we will keep fixed a random sample size $N_n$ such that
\ \\
- $N_n$ can take the values $k_1n, \ldots, k_{L + 1}n$,
\ \\
- for each $i \in \{1,\ldots, L\}$, the event $\{N_n = k_i n\}$ does not depend on $X_{k_i n + 1}, X_{k_i n + 2}, \ldots$,
\ \\
- for each $i \in \{1,\ldots,L\}$, the stopping rule
\begin{equation}
\mathbb{P}[N_n = k_i n \mid X_1, \ldots, X_{k_in }] = \psi\left(K_{k_i n}/(k_in)^\gamma\right) \prod_{j = 1}^{i - 1} (1 - \psi)\left(K_{k_j n} / (k_j n)^\gamma\right)\label{eq:StoppingRule}
\end{equation}
holds, where, for $m \in \mathbb{N}_0$, $K_{m} = \sum_{k = 1}^m X_k$, and the empty product is $1$.

For each $n \in \mathbb{N}_0$, the above mathematical model represents a group sequential trial in which the observations $X_1, X_2, \ldots$ are collected, and which has random length $N_n$. After having collected the observations $X_1,\ldots,X_{k_1 n}$, an interim analysis of the sum $K_{k_1 n} =  \sum_{k = 1}^{k_1 n} X_k$ is performed, and, based on the stopping rule (\ref{eq:StoppingRule}) with $i = 1$, one decides whether the trial is stopped, i.e. has length $k_1 n$, or continued, i.e. the additional observations $X_{k_1n + 1}, \ldots, X_{k_{2}n}$ are collected, and a new interim analysis of the sum $K_{k_2 n} =  \sum_{k = 1}^{k_2 n} X_k$ is performed, based on the stopping rule (\ref{eq:StoppingRule}) with $i = 2$. This process continues either until the trial is stopped at a point where an interim analysis is performed, or until the maximal length $k_{L + 1} n$ is reached. Thus the actual length $N_n$ can take the values $k_1n, k_2n, \ldots, k_{L+1}n$.

We wish to point out that the above mathematical model contains a very popular setting in sequential analysis. Indeed, if we let
\ \\
- for each $i \in \{1,\ldots,L+1\}$, $k_i = i$, 
\ \\
- $\psi$ be the map
\begin{equation}
\psi(x) = \left\{\begin{array}{clrr}      
1 &\textrm{ if }& \left|x\right| \geq C\\       
0 &\textrm{ if }&  \left|x\right| < C
\end{array}\right.,\label{eq:PsiEx}
\end{equation}
where $C \in \mathbb{R}^+_0$ is a fixed constant, 

then $N_n$ becomes a random sample size such that
\ \\
- $N_n$ can take the values $n, 2n, \ldots, (L+1)n$, 
\ \\
- for each $i \in \{1,\ldots, L\}$, the event $\{N_n = i n\}$ does not depend on $X_{in + 1}, X_{in + 2}, \ldots$, 
\ \\
- for each $i \in \{1,\ldots,L\}$, the stopping rule (\ref{eq:StoppingRule}) is transformed into

\begin{eqnarray}
\lefteqn{\mathbb{P}[N_n = i n \mid X_1, \ldots, X_{in }]}\label{eq:StoppingRuleSpecific}\\
 &=& \left\{\begin{array}{clrr}     
1 &\textrm{ if }& \left|K_{in} \right| \geq C (in)^\gamma \text{ and } \forall j \in \{1,\ldots,i-1\} : \left|K_{jn} \right| <  C (jn)^\gamma\\       
0 &\textrm{ otherwise }&  \nonumber 
\end{array}\right..\nonumber
\end{eqnarray}
So this situation corresponds to a group sequential trial in which independent observations with law $N(\mu,\sigma^2)$ are collected, and in which an interim analysis of the sum of the observations is performed after every $n$ observations. The stopping rule (\ref{eq:StoppingRuleSpecific}) indicates that the trial is stopped either at the first $i n$ for which $\left|K_{in}\right|$ exceeds $C (in)^\gamma$, or at $(L+ 1)n$. This example has been studied extensively in the literature, see e.g. \cite{S78}, \cite{C89}, and \cite{W92}. The number $\gamma$ is called the shape parameter. The choice $\gamma = 1/2$ is known as the Pocock paradigm, introduced in \cite{P77}, and the choice $\gamma = 0$ as the O'Brien-Fleming paradigm, introduced in \cite{OF79}.

We now return to the general setting, determined by the first two paragraphs of this section. We are interested in the asymptotic behavior of the sample mean $\widehat{\mu}_{N_n} = K_{N_n}/N_n$, used as an estimator for $\mu$, as $n \to \infty$. Notice that, $L$ and the $k_i$ being kept fixed, we are in fact looking at the asymptotics of a group sequential trial in which the maximal sample size $k_{L+1}n$ increases, but in which the number $L$ of interim analyses is fixed, and in which, for each $i \in \{1,\ldots,L\}$, the ratio of the number of observations collected at the $i$th interim analysis, i.e. $k_i n$, to the maximal number of observations, i.e. $k_{L + 1} n$, is a fixed number, i.e. $k_i/k_{L + 1}$.

Very recently, building upon \cite{MKA14}, the bias and the mean squared error (MSE) of the sample mean $\widehat{\mu}_{N_n}$ were studied in \cite{BIM}. From the general results obtained in that paper, it can be easily derived that $\widehat{\mu}_{N_n}$ is asymptotically unbiased in the sense that
$$\mathbb{E}\left[{\widehat{\mu}_{N_n}} - \mu\right] \to 0 \text{ as } n \to \infty \text{ with rate } 1/\sqrt{n},$$
and has a vanishing MSE in the sense that
$$\mathbb{E}\left[\left({\widehat{\mu}_{N_n}} - \mu\right)^2\right] \to 0 \text{ as } n \to \infty \text{ with rate } 1/n.$$
Moreover, the rates $1/\sqrt{n}$ and $1/n$ were shown to be optimal. 

Concerning asymptotic normality of $\widehat{\mu}_{N_n}$, first steps were taken in \cite{BIM}, but the authors were not able to prove a powerful result. In this paper, building upon \cite{BIM}, we will prove an asymptotic normality result for $\widehat{\mu}_{N_n}$ in the general setting. 

The paper is structured as follows. In section \ref{sec:Prelim}, we  formulate some preliminary results from \cite{BIM}. In section \ref{sec:FunEq}, a fundamental equation will be derived from the theory developed in \cite{BIM}. In section \ref{sec:ProofMain}, we provide the main result of this paper, which states that the quantity $\frac{\sqrt{N_n}}{\sigma}\left(\widehat{\mu}_{N_n} - \mu\right)$ converges to $N(0,1)$ for the total variation distance as $n \to \infty$ in each of the following cases:
\begin{itemize}
\item[(A)] $\gamma > 1$ and $\psi$ has a finite limit in $0$, 
\item[(B)] $\gamma = 1$ and $\psi$ has a finite limit in $\mu$, 
\item[(C)] (1)  $1/2 < \gamma < 1$,  $\mu \neq 0$, and $\psi$ has finite limits in $-\infty$ and $\infty$,
\item[\phantom{(c)}] (2) $1/2 < \gamma < 1$, $\mu = 0$, and $\psi$ has a finite limit in $0$,
\item[(D)] $\gamma = 1/2$, $\mu \neq 0$, and $\psi$ has finite limits in $-\infty$ and $\infty$,
\item[(E)] (1) $0 \leq \gamma < 1/2$, $\mu \neq 0$, and $\psi$ has finite limits in $- \infty$ and $\infty$,
\item[\phantom{(e)}] (2) $0 \leq \gamma < 1/2$, $\mu = 0$, and $\psi$ has coinciding finite limits in $- \infty$ and $\infty$.
\end{itemize}
Notice that the conditions imposed on $\psi$ are mild, and are always satisfied if $\psi$ takes the form (\ref{eq:PsiEx}) and $\mu \notin \{-C,C\}$. The proof of the main result is based on the fundamental equation formulated in section \ref{sec:FunEq}. We also state a conclusion for the validity of confidence intervals for the parameter $\mu$ based on $\widehat{\mu}_{N_n}$. In section \ref{sec:TEx}, some examples are given, which show that the conditions under which the main result holds, cannot be omitted. A simulation study is conducted in section \ref{sec:Sim}. Some concluding remarks are formulated in section \ref{sec:ConRem}.

\section{Preliminaries}\label{sec:Prelim}

Let $\phi$ be the standard normal density. For a finite tuple $B = (b_1,\ldots,b_i)$ of bounded Borel measurable maps of $\mathbb{R}$ into $\mathbb{R}$, the {\em normal transform} was defined in \cite{BIM} (section 2, first paragraph) as the map $\mathcal{N}_{B,\mu,\sigma}$ of $]0,\infty[^{i + 1} \times \mathbb{R}$ into $\mathbb{R}$ given by
\begin{eqnarray*}
\lefteqn{\mathcal{N}_{B,\mu,\sigma}(x_1,\ldots,x_{i+1},x)}\label{eq:NormalTransform}\\ 
&=& \frac{\int_{-\infty}^\infty \ldots \int_{-\infty}^\infty \prod_{j=1}^i \frac{\phi\left(\frac{z_j - \mu x_j}{\sigma \sqrt{x_j}}\right)}{\sigma \sqrt{x_j}} b_j\left(\sum_{k=1}^j z_k\right) \frac{\phi\left(\frac{x - \sum_{k=1}^i z_k - \mu x_{i + 1}}{\sigma \sqrt{x_{i + 1}}}\right)}{\sigma \sqrt{x_{i + 1}}} dz_i \ldots dz_1}{\frac{1}{\sigma \sqrt{\sum_{k=1}^{i + 1} x_k}}\phi\left(\frac{x - \mu\sum_{k=1}^{i + 1} x_k}{\sigma \sqrt{\sum_{k=1}^{i+1} x_k}}\right)}.\nonumber
 \end{eqnarray*}
 
 The following recursive description of the normal transform was obtained in \cite{BIM} (section 2, Theorem 3). 
 
 \begin{thm}\label{thm:RecursiveDescriptionNormalTransform}
Let $\xi$ be a random variable with law $N(0,1)$. For a bounded Borel measurable map $b$ of $\mathbb{R}$ into $\mathbb{R}$, $x_1,x_2 \in \left]0,\infty\right[$, and $x \in \mathbb{R}$,
\begin{equation}
\mathcal{N}_{b,\mu,\sigma}(x_1,x_2,x) =  \mathbb{E}\left[b\left(\frac{x_1}{x_1 + x_2} x + \sigma \sqrt{\frac{x_1 x_2}{x_1 + x_2}} \xi\right)\right].\label{eq:NormalTransform1}
\end{equation}
Furthermore, for a natural number $i \geq 2$, a tuple $(b_1,\ldots,b_i)$ of bounded Borel measurable maps of $\mathbb{R}$ into $\mathbb{R}$, $x_1,\ldots,x_{i+1} \in \left]0,\infty\right[$, and $x \in \mathbb{R}$,
\begin{equation}
\mathcal{N}_{(b_1,\ldots,b_{i}), \mu,\sigma}(x_1,\ldots,x_{i+1},x) = \mathcal{N}_{(b_1,\ldots,b_{i-2},\widetilde{b}_{i-1}),\mu,\sigma}(x_1,\ldots,x_{i-1},x_{i} + x_{i + 1},x),\label{eq:NormRec}
\end{equation}
where
\begin{equation}
\widetilde{b}_{i-1}(z) = b_{i-1}(z) \mathbb{E}\left[b_i\left(\frac{x_{i+1}}{x_i + x_{i+1}} z + \frac{x_i}{x_i + x_{i + 1}} x + \sigma \sqrt{\frac{x_i x_{i+1}}{x_i + x_{i+1}}} \xi\right)\right].\label{eq:bTilde}
\end{equation}
\end{thm}

For all $n \in \mathbb{N}_0$ and $i \in \{1,\ldots, L\}$, put
\begin{equation}
m_{n,i} = k_i n,\label{eq:mni}
\end{equation}
and 
\begin{equation}
M_n = k_{L + 1} n.\label{eq:Mn}
\end{equation}
Also, for all $n \in \mathbb{N}_0$ and $i \in  \{1,\ldots, L\}$, let $\psi_{m_{n,i}}$ be the Borel measurable map defined by
\begin{equation}
\psi_{m_{n,i}}(x) = \psi(x/ m_{n,i}^\gamma).\label{eq:specifyPsiMni}
\end{equation}
Then, by the second paragraph of section \ref{sec:Intro}, $N_n$ is a random sample size such that
\ \\
- $N_n$ can take the values $m_{n,1}, \ldots, m_{n,L}, M_n$,
\ \\
- for each $i \in \{1,\ldots, L\}$, the event $\{N_n = m_{n,i}\}$ does not depend on $X_{m_{n,i} + 1}, X_{m_{n,i} + 2}, \ldots$,
\ \\
- for each $i \in \{1,\ldots,L\}$, the stopping rule
\begin{equation*}
\mathbb{P}[N_n = m_{n,i} \mid X_1, \ldots, X_{m_{n,i}}] = \psi_{m_{n,i}}\left(K_{m_{n,i}}\right) \prod_{j = 1}^{i - 1} (1 - \psi_{m_{n,j}})\left(K_{m_{n,j}}\right)\label{eq:StoppingRuleBis}
\end{equation*}
holds. 

This brings us exactly in the setting of \cite{BIM} (section 1, first paragraph). Now, as in \cite{BIM} (section 3, first paragraph), put 
\begin{equation}
\Delta m_{n,1} = m_{n,1}\label{eq:Deltamn1}
\end{equation}
and
\begin{equation}
\mathcal{N}_{n,1}(x) = 1,\label{eq:N1}
\end{equation}
and, for $i \in \{2,\ldots,L\}$, 
\begin{equation}
\Delta_{m_{n,i}} = m_{n,i} - m_{n,i-1}\label{eq:DeltaMni}
\end{equation}
and
\begin{equation}
\mathcal{N}_{n,i}(x) = \mathcal{N}_{\left(1 - \psi_{m_{n,1}}, \ldots, 1 - \psi_{m_{n,i-1}}\right),\mu, \sigma}(\Delta_{m_{n,1}},\ldots,\Delta_{m_{n,i}},x).\label{eq:Ni}
\end{equation}
The following result was obtained in \cite{BIM} (section 7, Lemma 17). 

\begin{thm}\label{thm:FunEqNormTran}
Let $\xi$ and $\eta$ be independent random variables with law $N(0,1)$. Then, for each $n \in \mathbb{N}_0$ and each bounded Borel measurable test function $h : \mathbb{R} \rightarrow \mathbb{R}$,
\begin{equation}
\mathbb{E}\left[\left\{h(\xi) - h\left(\frac{\sqrt{N_n}}{\sigma}\left(\widehat{\mu}_{N_n} - \mu\right)\right)\right\}\right] = \sum_{i=1}^L \mathbb{E}\left[h(\xi)\left\{\mathcal{M}_i(n) - \widetilde{\mathcal{M}}_i(n)\right\} \right],\label{eq:FunEqNormTran}
\end{equation}
with
\begin{equation*}
\mathcal{M}_i(n) = \left(\psi_{m_{n,i}} \mathcal{N}_{n,i}\right)\left(\mu m_{n,i} + \sigma \sqrt{m_{n,i}} \xi\right)
\end{equation*}
and
\begin{equation*}
\widetilde{\mathcal{M}}_i(n) = \left(\psi_{m_{n,i}} \mathcal{N}_{n,i}\right)\left(\mu m_{n,i} + \sigma \sqrt{\frac{m_{n,i}(M_n - m_{n,i})}{M_n}} \xi + \sigma \frac{m_{n,i}}{\sqrt{M_n}} \eta\right).
\end{equation*}
\end{thm}

In the next section of this paper, we will combine the recursive description of the normal transform in Theorem \ref{thm:RecursiveDescriptionNormalTransform} with equation (\ref{eq:FunEqNormTran}) in Theorem \ref{thm:FunEqNormTran}, to obtain a fundamental equation.

\section{A fundamental equation}\label{sec:FunEq}

We will state a fundamental equation in Theorem \ref{thm:FunEq}. We first need two lemmas. 

\begin{lem}\label{lem:FirstLemRecNorm}
For each $i \in \{2,\ldots,L\}$, the following assertion holds. For each collection of real numbers $0 < x_1 < x_2 < \ldots < x_i$, there exist fixed continuous random variables $\rho_{i,1}, \rho_{i,2}, \ldots,\rho_{i,i-1}$, only depending on $\sigma$ and the $x_j$, such that, for each $n \in \mathbb{N}_0$, each collection $b_1,\ldots,b_{i-1}$ of bounded Borel measurable maps of $\mathbb{R}$ into $\mathbb{R}$, and each $x \in \mathbb{R}$,
\begin{equation}
\mathcal{N}_{(b_1,\ldots,b_{i-1}),\mu,\sigma}(x_1n, x_2n,\ldots,x_in,x) = \mathbb{E}\left[\prod_{j=1}^{i-1} b_j\left(\frac{\sum_{k=1}^j x_k}{\sum_{k = 1}^i x_k} x +  \rho_{i,j} \sqrt{n}\right)\right].\label{eq:firstlem}
\end{equation}
\end{lem}

\begin{proof}
Assume that $L \geq 2$. The proof is by induction on $i \in \{2,\ldots,L\}$, using Theorem \ref{thm:RecursiveDescriptionNormalTransform}.

Take $i = 2$, strictly positive real numbers $x_1,x_2$, $n \in \mathbb{N}_0$, a bounded Borel measurable map $b$ of $\mathbb{R}$ into $\mathbb{R}$, and $x \in \mathbb{R}$. Fix a standard normally distributed random variable $\xi$. Then, by (\ref{eq:NormalTransform1}), 
$$\mathcal{N}_{b_1,\mu,\sigma} (x_1 n , x_2 n, x) = \mathbb{E}\left[b_1\left(\frac{x_1}{x_1 + x_2} x + \sigma \sqrt{\frac{x_1 x_2}{x_1 + x_2}} \xi \sqrt{n}\right)\right].$$
Thus, putting 
$$\rho_{2,1} = \sigma \sqrt{\frac{x_1 x_2}{x_1 + x_2}} \xi,$$
shows that the desired assertion holds for $i = 2$.

Now assume that $L > 2$ and that the desired assertion holds for some $i \in \{2,\ldots,L-1\}$. We will prove that it also holds for $i + 1$. Take strictly positive real numbers $x_1, x_2, \ldots, x_{i + 1}$, $n \in \mathbb{N}_0$, bounded Borel measurable maps $b_1, \ldots, b_i$ of $\mathbb{R}$ into $\mathbb{R}$, and $x \in \mathbb{R}$. By (\ref{eq:NormRec}) and (\ref{eq:bTilde}),
\begin{equation}
\mathcal{N}_{(b_1,\ldots,b_{i}), \mu,\sigma}(x_1 n,\ldots,x_{i+1} n,x) = \mathcal{N}_{(b_1,\ldots,b_{i-2},\widetilde{b}_{i-1}),\mu,\sigma}(x_1 n ,\ldots,x_{i-1} n,(x_{i} + x_{i + 1}) n,x),\label{eq:NormRecUsed}
\end{equation}
where, for any standard normally distributed random variable $\xi$,
\begin{equation}
\widetilde{b}_{i-1}(z) = b_{i-1}(z) \mathbb{E}\left[b_i\left(\frac{x_{i+1}}{x_i + x_{i+1}} z + \frac{x_i}{x_i + x_{i + 1}} x + \sigma \sqrt{\frac{x_i x_{i+1}}{x_i + x_{i+1}}} \xi \sqrt{n}\right)\right].\label{eq:bTildeUsed}
\end{equation}
By the induction hypothesis applied to $x_1, \ldots, x_{i-1}, x_i + x_{i +1}$, there exist fixed continuous random variables $\rho_{i,1},\ldots,\rho_{i,i-1}$, only depending on $\sigma$ and the $x_j$, such that 
\begin{eqnarray}
\lefteqn{\mathcal{N}_{(b_1,\ldots,b_{i-2},\widetilde{b}_{i-1}),\mu,\sigma}(x_1 n ,\ldots,x_{i-1} n,(x_{i} + x_{i + 1}) n,x)}\label{eq:IHUsed}\\
&=& \mathbb{E}\left[\prod_{j=1}^{i-2} b_j\left(\frac{\sum_{k=1}^j x_k}{\sum_{k = 1}^{i + 1} x_k} x +  \rho_{i,j} \sqrt{n}\right) \widetilde{b}_{i - 1}\left(\frac{\sum_{k=1}^{i-1} x_k}{\sum_{k = 1}^{i + 1} x_k} x + \rho_{i,i-1} \sqrt{n}\right)\right]\nonumber.
\end{eqnarray}
But then, fixing a standard normally distributed random variable $\xi$, stochastically independent of $\rho_{i,1},\ldots,\rho_{i,i-1}$, and combining (\ref{eq:NormRecUsed}), (\ref{eq:bTildeUsed}) and (\ref{eq:IHUsed}), gives
\begin{eqnarray*}
\lefteqn{\mathcal{N}_{(b_1,\ldots,b_{i}), \mu,\sigma}(x_1 n,\ldots,x_{i+1} n,x)}\\
&=&  \mathbb{E}{\Huge [}\prod_{j=1}^{i-1} b_j\left(\frac{\sum_{k=1}^j x_k}{\sum_{k = 1}^{i + 1} x_k} x + \rho_{i,j} \sqrt{n}\right)\\
&& b_i\left(\frac{x_{i+1}}{x_i + x_{i+1}} \left(\frac{\sum_{k=1}^{i-1} x_k}{\sum_{k = 1}^{i + 1} x_k} x +  \rho_{i,i-1} \sqrt{n}\right) + \frac{x_i}{x_i + x_{i + 1}} x + \sigma \sqrt{\frac{x_i x_{i+1}}{x_i + x_{i+1}}} \xi \sqrt{n}\right){\Huge]}\\
&=& \mathbb{E}{\Huge [}\prod_{j=1}^{i-1} b_j\left(\frac{\sum_{k=1}^j x_k}{\sum_{k = 1}^{i + 1} x_k} x + \rho_{i,j} \sqrt{n}\right)\\
&& b_i\left(\frac{\sum_{k = 1}^i x_i}{\sum_{k = 1}^{i + 1} x_k} x + \left(\frac{x_{i+1}}{x_i + x_{i+1}} \rho_{i,i-1} + \sigma \sqrt{\frac{x_i x_{i+1}}{x_i + x_{i+1}}} \xi \right) \sqrt{n} \right){\Huge]}.
\end{eqnarray*}
Thus, considering the collection of continuous random variables $\widetilde{\rho}_{i+1,1}, \ldots, \widetilde{\rho}_{i+1,i}$,  defined by
$$\forall j \in \{1,\ldots,i-1\} : \widetilde{\rho}_{i+1,j} = \rho_{i,j} $$
and
$$\widetilde{\rho}_{i + 1, i} = \frac{x_{i+1}}{x_i + x_{i+1}} \rho_{i,i-1} + \sigma \sqrt{\frac{x_i x_{i+1}}{x_i + x_{i+1}}} \xi,$$
shows that the desired assertion holds for $i + 1$. 

This finishes the proof.
\end{proof}

\begin{lem}\label{lem:SecondLemRecNorm}
Fix everything as in the first two paragraphs of section 1. Then, for each $i \in \{2,\ldots,L\}$, there exist fixed continuous random variables $\sigma_{i,1}, \sigma_{i,2}, \ldots,\sigma_{i,i-1}$, only depending on $\gamma$, $\sigma$ and the $k_j$, such that, for all $n \in \mathbb{N}_0$ and $x \in \mathbb{R}$,  
$$\mathcal{N}_{n,i}(x) = \mathbb{E}\left[\prod_{j = 1}^{i - 1} (1 - \psi) \left(\frac{k_j^{1 - \gamma}}{k_i} n^{-\gamma} x + \sigma_{i,j} n^{\frac{1}{2} - \gamma}\right)\right].$$
\end{lem}

\begin{proof}
Applying Lemma \ref{lem:FirstLemRecNorm} to the strictly positive numbers $k_1, k_2 - k_1, \ldots, k_i - k_{i-1}$, reveals the existence of fixed continuous random variables $\rho_{i,1},\ldots,\rho_{i,i-1}$, only depending on $\sigma$ and the $k_j$, such that, by (\ref{eq:firstlem}), keeping the notations (\ref{eq:mni}), (\ref{eq:Deltamn1}), (\ref{eq:DeltaMni}) and (\ref{eq:Ni}) in mind, 
$$\mathcal{N}_{n,i}(x) = \mathbb{E}\left[\prod_{j=1}^{i-1} \left(1 - \psi_{m_{n,j}}\right)\left(\frac{k_j}{k_i} x +  \rho_{i,j} \sqrt{n}\right)\right],$$
which, by (\ref{eq:specifyPsiMni}),
$$= \mathbb{E}\left[\prod_{j=1}^{i-1} \left(1 - \psi\right)\left(\frac{k_j^{1 - \gamma}}{k_i} n^{-\gamma} x +  \frac{\rho_{i,j}}{k_j^\gamma} n^{\frac{1}{2} - \gamma}\right)\right],$$
which, putting $\sigma_{i,j} = \frac{\rho_{i,j}}{k_j^\gamma}$, proves the desired result.
\end{proof}

\begin{thm}\label{thm:FunEq}
Fix everything as in the first two paragraphs of section \ref{sec:Intro}. Then there exist fixed continuous random variables $\alpha_1, \ldots, \alpha_{L}$ and $\gamma_1,\ldots, \gamma_{L}$, only depending on $\gamma$, $\sigma$ and the $k_j$, and, for all $i \in \{2,\ldots,L\}$ and $j \in \{1,\ldots,i-1\}$, fixed continuous random variables $\beta_{i,j}$ and $\delta_{i,j}$, only depending on $\gamma$, $\sigma$ and the $k_j$, and a standard normally distributed random variable $\xi$, such that, for each $n \in \mathbb{N}_0$ and each bounded Borel measurable test function $h : \mathbb{R} \rightarrow \mathbb{R}$, 
\begin{equation*}
\mathbb{E}\left[\left\{h(\xi) - h\left(\frac{\sqrt{N_n}}{\sigma}\left(\widehat{\mu}_{N_n} - \mu\right)\right)\right\}\right] = \sum_{i=1}^{L} \mathbb{E}\left[h(\xi) \left\{\mathcal{R}_i(n) - \widetilde{\mathcal{R}}_i(n)\right\}\right],\label{eq:FunEq}
\end{equation*}
with
\begin{equation*}
\mathcal{R}_i(n) = \psi\left(k_i^{1 - \gamma} n^{1 - \gamma} \mu + \alpha_i n^{\frac{1}{2} - \gamma}\right) \prod_{j = 1}^{i-1} (1 - \psi)\left(k_{j}^{1-\gamma} n^{1 - \gamma}\mu +\beta_{i,j} n^{\frac{1}{2} - \gamma}\right) 
\end{equation*}
and
\begin{equation*}
\widetilde{\mathcal{R}}_i(n) = \psi\left(k_i^{1 - \gamma} n^{1 - \gamma} \mu + \gamma_i n^{\frac{1}{2} - \gamma}\right) \prod_{j = 1}^{i-1} (1 - \psi)\left(k_{j}^{1-\gamma} n^{1 - \gamma}\mu + \delta_{i,j}n^{\frac{1}{2} - \gamma} \right),
\end{equation*}
the empty product being $1$.
\end{thm}

\begin{proof}
By Lemma \ref{lem:SecondLemRecNorm}, for each $i \in \{2,\ldots,L\}$, there exists a fixed collection $\sigma_{i,1}, \sigma_{i,2}, \ldots,\sigma_{i,i-1}$ of continuous random variables, only depending on $\gamma$, $\sigma$ and the $k_j$, such that, for all $n \in \mathbb{N}_0$ and $x \in \mathbb{R}$,  
\begin{equation}
\mathcal{N}_{n,i}(x) = \mathbb{E}\left[\prod_{j = 1}^{i - 1} (1 - \psi) \left(\frac{k_j^{1 - \gamma}}{k_i} n^{-\gamma} x + \sigma_{i,j} n^{\frac{1}{2} - \gamma}\right)\right].\label{eq:FormNni}
\end{equation}
Then fix stochastically independent standard normally distributed random variables $\xi$ and $\eta$, each stochastically independent of all the $\sigma_{i,j}$. By Theorem \ref{thm:FunEqNormTran},
\begin{equation}
\mathbb{E}\left[\left\{h(\xi) - h\left(\frac{\sqrt{N_n}}{\sigma}\left(\widehat{\mu}_{N_n} - \mu\right)\right)\right\}\right] = \sum_{i=1}^L \mathbb{E}\left[h(\xi)\left\{\mathcal{M}_i(n) - \widetilde{\mathcal{M}}_i(n)\right\} \right],\label{eq:FunEqNormTranUsed}
\end{equation}
with
\begin{equation}
\mathcal{M}_i(n) = \left(\psi_{m_{n,i}} \mathcal{N}_{n,i}\right)\left(\mu m_{n,i} + \sigma \sqrt{m_{n,i}} \xi\right)\label{eq:MinUsed}
\end{equation}
and
\begin{equation}
\widetilde{\mathcal{M}}_i(n) = \left(\psi_{m_{n,i}} \mathcal{N}_{n,i}\right)\left(\mu m_{n,i} + \sigma \sqrt{\frac{m_{n,i}(M_n - m_{n,i})}{M_n}} \xi + \sigma \frac{m_{n,i}}{\sqrt{M_n}} \eta\right)\label{eq:MinTildeUsed}.
\end{equation}
Combining (\ref{eq:N1}), (\ref{eq:FormNni}), (\ref{eq:MinUsed}), (\ref{eq:mni}) and (\ref{eq:specifyPsiMni}), reveals that, for each $i \in \{1,\ldots,L\}$,
$$\mathbb{E}\left[h(\xi)\mathcal{M}_i(n)\right]$$
can be written as the expected value of the product with factors
\begin{eqnarray}
\lefteqn{h(\xi)\psi_{m_{n,i}}(\mu m_{n,i} + \sigma \sqrt{m_{n,i}} \xi)}\nonumber\\
&=& h(\xi)\psi\left(m_{n,i}^{1 - \gamma} \mu + \sigma m_{n,i}^{\frac{1}{2} - \gamma}\right)\nonumber\\
&=& h(\xi)\psi\left(k_i^{1 - \gamma} n^{1 - \gamma} \mu + \sigma k_i^{\frac{1}{2} - \gamma} \xi n^{\frac{1}{2} - \gamma}\right)\label{eq:Factor1}
\end{eqnarray}
and
\begin{eqnarray}
\lefteqn{\prod_{j = 1}^{i - 1} (1 - \psi) \left(\frac{k_j^{1 - \gamma}}{k_i} n^{-\gamma} \left(\mu m_{n,i} + \sigma \sqrt{m_{n,i}} \xi\right) + \sigma_{i,j} n^{\frac{1}{2} - \gamma}\right)}\nonumber\\
&=& \prod_{j = 1}^{i - 1} (1 - \psi) \left(k_j^{1 - \gamma} n^{1 - \gamma} \mu +  \left( \sigma \frac{k_j^{1 - \gamma}}{\sqrt{k_i}} \xi + \sigma_{i,j}\right) n^{\frac{1}{2} - \gamma}\right),\label{eq:Factor2}
\end{eqnarray}
the empty product being $1$.
Analogously, combining (\ref{eq:N1}), (\ref{eq:FormNni}), (\ref{eq:MinTildeUsed}), (\ref{eq:mni}), (\ref{eq:Mn}) and (\ref{eq:specifyPsiMni}), shows that, for each $i \in \{1,\ldots,L\}$,
$$\mathbb{E}\left[h(\xi)\widetilde{\mathcal{M}}_i(n)\right]$$
can be written as the expected value of the product with factors
\begin{eqnarray}
\lefteqn{h(\xi)\psi_{m_{n,i}}\left(\mu m_{n,i} + \sigma \sqrt{\frac{m_{n,i}(M_n - m_{n,i})}{M_n}} \xi + \sigma \frac{m_{n,i}}{\sqrt{M_n}} \eta\right)}\nonumber\\
&=& h(\xi) \psi\left(m_{n,i}^{1-\gamma} \mu + \sigma \sqrt{\frac{M_{n} - m_{n,i}}{M_{n}}} m_{n,i}^{\frac{1}{2} - \gamma} \xi + \sigma \sqrt{\frac{m_{n,i}}{M_n}} m_{n,i}^{\frac{1}{2} - \gamma} \eta\right)\nonumber\\
&=& h(\xi) \psi\left(k_i^{1-\gamma} n^{1 - \gamma}\mu + \sigma k_i^{\frac{1}{2} - \gamma}\left(\sqrt{\frac{k_{L + 1} - k_i}{k_{L + 1}}} \xi +  \sqrt{\frac{k_i}{k_{L + 1}}} \eta\right)  n^{\frac{1}{2} - \gamma}\right)\label{eq:Factor3}
\end{eqnarray}
and
\begin{eqnarray}
\lefteqn{\prod_{j = 1}^{i - 1} (1 - \psi) \left(\frac{k_j^{1 - \gamma}}{k_i} n^{-\gamma} \left(\mu m_{n,i} + \sigma \sqrt{\frac{m_{n,i}(M_n - m_{n,i})}{M_n}} \xi + \sigma \frac{m_{n,i}}{\sqrt{M_n}} \eta\right) + \sigma_{i,j} n^{\frac{1}{2} - \gamma}\right)}\nonumber\\
&=& \prod_{j = 1}^{i - 1} (1 - \psi) \left(k_j^{1 - \gamma} n^{1 - \gamma} \mu +  \left\{ \sigma \frac{k_j^{1 - \gamma}}{\sqrt{k_i}} \left(\sqrt{\frac{k_{L + 1} - k_i}{k_{L + 1}}} \xi + \sqrt{\frac{k_i}{k_{L + 1}}} \eta\right) + \sigma_{i,j}\right\} n^{\frac{1}{2} - \gamma}\right),\label{eq:Factor4}
\end{eqnarray}
the empty product being $1$.
Now, for $i \in \{1,\ldots,L\}$, put in (\ref{eq:Factor1})
$$\alpha_i = \sigma k_i^{\frac{1}{2} - \gamma} \xi,$$ 
for $i \in \{2,\ldots,L\}$ and $j \in \{1,\ldots,i-1\}$, put in (\ref{eq:Factor2})
$$\beta_{i,j} = \sigma \frac{k_j^{1 - \gamma}}{\sqrt{k_i}} \xi + \sigma_{i,j},$$ 
for $i \in \{1,\ldots,L\}$, put in (\ref{eq:Factor3})
$$\gamma_i = \sigma k_i^{\frac{1}{2} - \gamma}\left(\sqrt{\frac{k_{L + 1} - k_i}{k_{L + 1}}} \xi +  \sqrt{\frac{k_i}{k_{L + 1}}} \eta\right),$$
and, for $i \in \{2,\ldots,L\}$ and $j \in \{1,\ldots,i-1\}$, put in (\ref{eq:Factor4})
$$\delta_{i,j} = \sigma \frac{k_j^{1 - \gamma}}{\sqrt{k_i}} \left(\sqrt{\frac{k_{L + 1} - k_i}{k_{L + 1}}} \xi + \sqrt{\frac{k_i}{k_{L + 1}}} \eta\right) + \sigma_{i,j}.$$
This allows us to conclude from (\ref{eq:FunEqNormTranUsed}) that the desired equation holds.
\end{proof}

In the next section, we will derive an asymptotic normality result for $\widehat{\mu}_{N_n}$ from Theorem \ref{thm:FunEq}.

\section{The main result}\label{sec:ProofMain}

Recall that the total variation distance between random variables $X$ and $Y$ is given by
$$d_{TV}(X,Y) = \sup_{A} \left|\mathbb{P}[X \in A] - \mathbb{P}[Y \in A]\right|,$$
the supremum running over all Borel sets $A \subset \mathbb{R}$. It is well known that convergence for the total variation distance is strictly stronger than weak convergence. For more information on this distance, we refer the reader to \cite{R91} and \cite{Z83}.

We are now in a position to state and prove the main result of this paper.

\begin{thm}\label{thm:AsymNorSeq}
Fix everything as in the first two paragraphs of section \ref{sec:Intro}. Furthermore, consider a standard normally distributed random variable $\xi$ and the sample mean $\widehat{\mu}_{N_n} = K_{N_n}/N_n$. Then 
\begin{equation}
d_{TV}\left(\xi,\frac{\sqrt{N_n}}{\sigma}\left(\widehat{\mu}_{N_n} - \mu\right)\right) \to 0 \text{ as } n \to \infty\label{eq:dtvGoesToZero}
\end{equation}
in each of the following cases:
\begin{itemize}
\item[(A)] $\gamma > 1$ and $\psi$ has a finite limit in $0$, 
\item[(B)] $\gamma = 1$ and $\psi$ has a finite limit in $\mu$, 
\item[(C)] (1)  $1/2 < \gamma < 1$,  $\mu \neq 0$, and $\psi$ has finite limits in $-\infty$ and $\infty$,
\item[\phantom{(c)}] (2) $1/2 < \gamma < 1$, $\mu = 0$, and $\psi$ has a finite limit in $0$,
\item[(D)] $\gamma = 1/2$, $\mu \neq 0$, and $\psi$ has finite limits in $-\infty$ and $\infty$,
\item[(E)] (1) $0 \leq \gamma < 1/2$, $\mu \neq 0$, and $\psi$ has finite limits in $- \infty$ and $\infty$,
\item[\phantom{(e)}] (2) $0 \leq \gamma < 1/2$, $\mu = 0$, and $\psi$ has coinciding finite limits in $- \infty$ and $\infty$.
\end{itemize}
In particular, letting $\Phi$ be the standard normal cdf, it holds for each $x \in \mathbb{R}^+_0$ that
\begin{equation}
\left|2 \Phi (x) - 1 - \mathbb{P}\left[\widehat{\mu}_{N_n} - \frac{\sigma}{\sqrt{N_n}} x \leq \mu \leq \widehat{\mu}_{N_n} + \frac{\sigma}{\sqrt{N_n}} x \right]\right| \rightarrow 0 \textrm{ as } n \rightarrow \infty\label{eq:CIcon}
\end{equation}
in each of the cases (A)--(E), i.e. confidence intervals for $\mu$ based on $\widehat{\mu}_{N_n}$ are often reliable if $n$ is large enough.
\end{thm}

\begin{rem}
Our stopping rule (\ref{eq:StoppingRule}) is designed in such a way that it analyzes the sum of the observations based on the null hypothesis $\mu = 0$. This becomes clear if we take $\psi$ of the form (\ref{eq:PsiEx}). Indeed, in this case, the trial is stopped (to reject the null hypothesis $\mu = 0$) at the first $k_i n$ for which $\left|K_{k_in}\right| \geq C (k_i n)^\gamma$. Therefore, the distinction between $\mu = 0$ and $\mu \neq 0$ in Theorem \ref{thm:AsymNorSeq} is natural. Of course, it is possible to effortlessly switch to the null hypothesis $\mu = \mu_0$ for a different parameter $\mu_0 \in \mathbb{R}_0$, by considering the data $X_1 - \mu_0, X_2 - \mu_0, \ldots$
\end{rem}

\begin{rem}
In the case where $\gamma = 1/2$ and $\mu = 0$, we have not stated an asymptotic normality result. In the next section, we will show that such a result fails to hold (Example \ref{vb:Gamma1/2}). This is not a severe drawback, because the case $\gamma=1/2$ corresponds to the Pocock boundary, which is very rarely used in practice, often in favor of the more popular O'Brien-Fleming boundary, because it tends to assign too much type I error at early looks. In other words, it too generously allows for very early stopping (\cite{W87}).
\end{rem}

\begin{proof}[Proof of Theorem \ref{thm:AsymNorSeq}]
By Theorem \ref{thm:FunEq}, fix continuous random variables $\alpha_1, \ldots, \alpha_{L}$ and $\gamma_1,\ldots, \gamma_{L}$, only depending on $\gamma$, $\sigma$ and the $k_j$, and, for all $i \in \{2,\ldots,L\}$ and $j \in \{1,\ldots,i-1\}$, continuous random variables $\beta_{i,j}$ and $\delta_{i,j}$, only depending on $\gamma$, $\sigma$ and the $k_j$, and a standard normally distributed random variable $\xi_0$, such that, for each $n \in \mathbb{N}_0$ and each Borel measurable set $A \subset \mathbb{R}$,
\begin{equation}
\mathbb{E}\left[\left\{1_A(\xi_0) - 1_A\left(\frac{\sqrt{N_n}}{\sigma}\left(\widehat{\mu}_{N_n} - \mu\right)\right)\right\}\right] = \sum_{i=1}^{L} \mathbb{E}\left[ 1_A(\xi_0) \left\{\mathcal{R}_i(n) - \widetilde{\mathcal{R}}_i(n)\right\}\right],\label{eq:FunEqUse}
\end{equation}
with
\begin{equation}
\mathcal{R}_i(n) = \psi\left(k_i^{1 - \gamma} n^{1 - \gamma} \mu + \alpha_i n^{\frac{1}{2} - \gamma}\right) \prod_{j = 1}^{i-1} (1 - \psi)\left(k_{j}^{1-\gamma} n^{1 - \gamma}\mu +\beta_{i,j} n^{\frac{1}{2} - \gamma}\right) \label{eq:RinUse}
\end{equation}
and
\begin{equation}
\widetilde{\mathcal{R}}_i(n) = \psi\left(k_i^{1 - \gamma} n^{1 - \gamma} \mu + \gamma_i n^{\frac{1}{2} - \gamma}\right) \prod_{j = 1}^{i-1} (1 - \psi)\left(k_{j}^{1-\gamma} n^{1 - \gamma}\mu + \delta_{i,j}n^{\frac{1}{2} - \gamma} \right),\label{eq:RinTildeUse}
\end{equation}
the empty product being $1$. Since $\xi$ and $\xi_0$ have the same distribution, we infer from (\ref{eq:FunEqUse}) that the upper bound
\begin{equation}
d_{TV}\left(\xi,\frac{\sqrt{N_n}}{\sigma}\left(\widehat{\mu}_{N_n} - \mu\right)\right) \leq \sum_{i=1}^{L} \mathbb{E}\left[ \left| \mathcal{R}_i(n) - \widetilde{\mathcal{R}}_i(n) \right|\right]\label{eq:dtvUpperBound}
\end{equation}
holds for each $n \in \mathbb{N}_0$. The fact that limit relation (\ref{eq:dtvGoesToZero}) holds in each of the cases (A)--(E), can be easily read off from the upper bound (\ref{eq:dtvUpperBound}) and the expressions (\ref{eq:RinUse}) and (\ref{eq:RinTildeUse}). 

Suppose, for instance, that we are in the case (E) (2), i.e. $0 \leq \gamma < 1/2$, $\mu = 0$, and $\psi$ has coinciding finite limits in $- \infty$ and $\infty$, say with value $\lambda$. The $\alpha_i$, $\beta_{i,j}$, $\gamma_i$ and $\delta_{i,j}$ being continuous, they all differ from zero outside a fixed event with probability zero. Thus, almost surely, each of the expressions $\alpha_i n^{\frac{1}{2} - \gamma}$, $\beta_{i,j} n^{\frac{1}{2} - \gamma}$, $\gamma_i n^{\frac{1}{2} - \gamma}$ and $\delta_{i,j}n^{\frac{1}{2} - \gamma}$ tends to either $-\infty$ or $\infty$ as $n \to \infty$, whence, by (\ref{eq:RinUse}) and (\ref{eq:RinTildeUse}), almost surely, $\mathcal{R}_i(n)$ and $\widetilde{\mathcal{R}}_i(n)$ both tend to $\lambda (1 - \lambda)^{i-1}$ as $n \to \infty$. We conclude from the dominated convergence theorem that the upper bound in (\ref{eq:dtvUpperBound}) tends to $0$.

The other cases are analyzed analogously.

Finally, for each $x \in \mathbb{R}^+_0$, considering the Borel set $A = [-x,x]$, allows us to derive (\ref{eq:CIcon}) from (\ref{eq:dtvGoesToZero}), which finishes the proof.
\end{proof}

\begin{gev}\label{gev:MainCor}
Fix everything as in the first two paragraphs of section \ref{sec:Intro}, and assume in addition that $\psi$ takes the form (\ref{eq:PsiEx}) and $\mu \notin \{-C,C\}$. Then always 
$$d_{TV}\left(N(0,1),\frac{\sqrt{N_n}}{\sigma}\left(\widehat{\mu}_{N_n} - \mu\right)\right) \to 0 \text{ as } n \to \infty,$$
except in the case where $\gamma = 1/2$ and $\mu = 0$. In particular, letting $\Phi$ be the standard normal cdf, it always holds that, for each $x \in \mathbb{R}^+_0$,
$$\left|2 \Phi (x) - 1 - \mathbb{P}\left[\widehat{\mu}_{N_n} - \frac{\sigma}{\sqrt{N_n}} x \leq \mu \leq \widehat{\mu}_{N_n} + \frac{\sigma}{\sqrt{N_n}} x \right]\right| \rightarrow 0 \text{ as } n \rightarrow \infty,$$
except in the case where $\gamma = 1/2$ and $\mu = 0$.
\end{gev}

\begin{proof}
This is an easy application of Theorem \ref{thm:AsymNorSeq}.
\end{proof}

Corollary \ref{gev:MainCor} shows that in the important situation of a group sequential trial with outcomes with law $N(\mu,\sigma^2)$, which is stopped either at the first $k_i n$ for which $\left|K_{k_i n}\right| \geq C (k_i n)^\gamma$ or at $k_{L + 1} n$, naive confidence intervals for $\mu$, based on the ordinary sample mean $\widehat{\mu}_{N_n}$, are reliable for large $n$ if $\mu \notin \{-C,C\}$ and $(\gamma,\mu ) \neq (1/2,0)$. We will illustrate this with simulations in section \ref{sec:Sim}.

In the next section, we will provide some examples, which show that the conditions under which Theorem \ref{thm:AsymNorSeq} holds, cannot be omitted.
\section{Some examples}\label{sec:TEx}

In this section, we take $\mu = 0$, $\sigma = 1$, $L = 1$, $k_1 = 1$ and $k_2 = 2$. Furthermore, we fix independent random variables $\xi$ and $\eta$ with law $N(0,1)$. Then it follows from Theorem \ref{thm:FunEqNormTran}, and (\ref{eq:mni}), (\ref{eq:Mn}) and (\ref{eq:specifyPsiMni}), that, for each $n \in \mathbb{N}_0$ and each bounded Borel measurable test function $h : \mathbb{R} \rightarrow \mathbb{R}$,
\begin{eqnarray}
\lefteqn{\mathbb{E}\left[\left\{h(\xi) - h\left(\sqrt{N_n}\left(\widehat{\mu}_{N_n} - \mu\right)\right)\right\}\right]}\label{eq:ExampleBasis}\\
&=&\mathbb{E}\left[h(\xi) \left\{\psi\left(\xi n^{\frac{1}{2} - \gamma} \right) - \psi\left(\frac{\sqrt{2}}{2}\left( \xi + \eta \right) n^{\frac{1}{2} - \gamma}\right)\right\}\right].\nonumber
\end{eqnarray}

Recall that the Kolmogorov distance between random variables $X$ and $Y$ is given by
$$K(X,Y) = \sup_{x \in \mathbb{R}} \left|\mathbb{P}\left[X \leq x\right] - \mathbb{P}\left[Y \leq x\right]\right|.$$
It is well known that, for a continuous random variable $X$ and a sequence of random variables $(Y_n)_n$, $K(X,Y_n) \rightarrow 0$ if and only if $(Y_n)_n$ converges weakly to $X$. For more information on this distance, we refer the reader to \cite{R91} and \cite{Z83}.

We will now derive two interesting examples from equation (\ref{eq:ExampleBasis}).

\begin{vb}\label{vb:Gamma1/2}
Take $\gamma = 1/2$ and, for $C \in \mathbb{R}^+_0$, $\psi$ of the form (\ref{eq:PsiEx}). Thus, for $n \in \mathbb{N}_0$, this example corresponds to a trial with independent standard normally distributed outcomes, of maximal length $2n$, and in which one interim analysis is performed after having collected $n$ data. The trial is stopped if the absolute value of the sum of the first $n$ data exceeds $C n^{1/2}$, and continued otherwise. Let, for $x \in \mathbb{R}$, $1_{\left]-\infty,x\right]}$ be the characteristic function of the set $\left]-\infty,x\right]$. Then, $\Phi$ being the standard normal cdf, a straightforward calculation allows us to derive from (\ref{eq:ExampleBasis}) that, for each $x \in \mathbb{R}$ and each $n \in \mathbb{N}_0$,
\begin{eqnarray}
\lefteqn{\mathbb{E}\left[\left\{1_{\left]-\infty,x\right]}(\xi) - 1_{\left]-\infty,x\right]}\left(\sqrt{N_n}\left(\widehat{\mu}_{N_n} - \mu\right)\right)\right\}\right]}\label{eq:ExampleBasisFirstUsed}\\
&=& \mathbb{E}\left[1_{\left[-\infty,x\right[}(\xi) 1_{\left]-C,C\right[}\left(\frac{\sqrt{2}}{2} (\xi + \eta)\right)\right] - \mathbb{E}\left[1_{\left]-\infty,x\right]}(\xi) 1_{\left]-C,C\right[}(\xi)\right]\nonumber\\
&=& \mathbb{E}\left[\Phi\left(\left(\sqrt{2} C - \eta\right) \wedge x\right) - \Phi\left( \left(- \sqrt{2} C - \eta \right) \wedge x\right)\right] - \left[\Phi\left( C \wedge x \right) - \Phi\left(- C \wedge x\right)\right]\nonumber,
\end{eqnarray}
$\wedge$ standing for the operation that takes the minimum of two numbers. In particular, taking $x = - C$ in (\ref{eq:ExampleBasisFirstUsed}), it holds for each $n \in \mathbb{N}_0$ that
\begin{eqnarray*}
\lefteqn{\mathbb{E}\left[\left\{1_{\left]-\infty,-C\right]}(\xi) - 1_{\left]-\infty,-C\right]}\left(\sqrt{N_n}\left(\widehat{\mu}_{N_n} - \mu\right)\right)\right\}\right]}\label{eq:KolUpAux}\\
&=& \mathbb{E}\left[\Phi\left(\left(\sqrt{2} C - \eta\right) \wedge -C\right) - \Phi\left(\left( - \sqrt{2} C - \eta \right)\wedge -C\right)\right],
\end{eqnarray*}
which, considering the event $\left\{0 \leq \eta \leq \sqrt{2} C\right\}$,
\begin{eqnarray*}
&\geq& \mathbb{P}\left[0 \leq \eta \leq \sqrt{2} C\right] \left[\Phi(- C) - \Phi\left(-\sqrt{4} C\right)\right]\nonumber\\
&=& \left[\Phi\left(\sqrt{2} C\right) - \Phi(0)\right] \left[\Phi\left(-C\right) - \Phi\left(-\sqrt{4} C\right)\right] > 0,\nonumber
\end{eqnarray*}
from which we conclude that
\begin{eqnarray*}
\lefteqn{K\left(\xi,\sqrt{N_n}\left(\widehat{\mu}_{N_n} - \mu\right)\right)}\\
&=& \sup_{x \in \mathbb{R}} \left|\mathbb{E}\left[\Phi\left(\left(\sqrt{2} C - \eta\right) \wedge x\right) - \Phi\left(\left( - \sqrt{2} C - \eta \right)\wedge x\right)\right] - \left[\Phi\left( C \wedge x \right) - \Phi\left(- C \wedge x\right)\right]\right|\nonumber\\
&\geq& \left[\Phi\left(\sqrt{2} C\right) - \Phi(0)\right] \left[\Phi\left(-C\right) - \Phi\left(-\sqrt{4} C\right)\right] > 0.\nonumber
\end{eqnarray*}
That is, for each $n \in \mathbb{N}_0$, the Kolmogorov distance between $\xi$ and $\sqrt{N_n}\left(\widehat{\mu}_{N_n} - \mu\right)$ takes a strictly positive constant value. We infer that $\left(\sqrt{N_n}\left(\widehat{\mu}_{N_n} - \mu\right)\right)_n$ fails to converge weakly to $\xi$. In particular, case (D) of Theorem \ref{thm:AsymNorSeq}  is not extendable to $\gamma = 1/2$ and $\mu = 0$, even if $\psi$ has a finite limit in $0$ and coinciding finite limits in $-\infty$ and $\infty$.
\end{vb}

\begin{vb}\label{vb:Psi0}
Take $\gamma \in \mathbb{R}^+_0$ arbitrary and $\psi = 1_{\left]-\infty,0\right]}$, the characteristic function of the set $\left]-\infty,0\right]$. Thus, for $n \in \mathbb{N}_0$, this example corresponds to a trial with independent standard normally distributed outcomes, of maximal length $2n$, and in which one interim analysis is performed after having collected $n$ data. The trial is stopped if the sum of the first $n$ data is negative, and continued otherwise. Let, for $x \in \mathbb{R}$, $1_{\left]-\infty,x\right]}$ be the characteristic function of the set $\left]-\infty,x\right]$. Then, $\Phi$ being the standard normal cdf, a straightforward calculation allows us to derive from (\ref{eq:ExampleBasis}) that, for each $x \in \mathbb{R}$ and each $n \in \mathbb{N}_0$,
\begin{eqnarray*}
\lefteqn{\mathbb{E}\left[\left\{1_{\left]-\infty,x\right]}(\xi) - 1_{\left]-\infty,x\right]}\left(\sqrt{N_n}\left(\widehat{\mu}_{N_n} - \mu\right)\right)\right\}\right]}\label{eq:ExampleBasisSecondUsed}\\
&=& \mathbb{E}\left[1_{\left]-\infty,x\right]}(\xi) 1_{\left]-\infty,0\right]}(\xi)\right]  - \mathbb{E}\left[1_{\left]-\infty,x\right]}(\xi) 1_{\left]-\infty,0\right]}\left(\xi + \eta\right)\right] \\
&=& \Phi(x \wedge 0) - \mathbb{E}\left[1_{\left]-\infty,x\right]}(\xi)\Phi(-\xi)\right]\\
&=& \Phi(x \wedge 0) - \Phi(x) + \Phi^2(x)/2, 
\end{eqnarray*}
$\wedge$ standing for the operation that takes the minimum of two numbers. We deduce that, for each $n \in \mathbb{N}_0$,
\begin{eqnarray*}
K\left(\xi,\sqrt{N_n}\left(\widehat{\mu}_{N_n} - \mu\right)\right) = \sup_{x \in \mathbb{R}} \left| \Phi(x \wedge 0) - \Phi(x) + \Phi^2(x)/2 \right|  = 1/8.
\end{eqnarray*}
We infer that $\left(\sqrt{N_n}\left(\widehat{\mu}_{N_n} - \mu\right)\right)_n$ fails to converge weakly to $\xi$. This shows that in the cases (A), (B), (C) (2) and (E) (2) of Theorem \ref{thm:AsymNorSeq}, the conditions on $\psi$ cannot be omitted.
\end{vb}

\section{Simulations}\label{sec:Sim}

In order to illustrate Theorem \ref{thm:AsymNorSeq} and Corollary \ref{gev:MainCor}, we have conducted the following simulation study. 

For $\mu \in \{-1,0,1\}$, $n \in \{10,50,100,500\}$, $C \in \{0,1,2\}$, and $\gamma \in \{0,0.25,0.5,0.75,1,2\}$, we have simulated a group sequential trial with the following properties.

\begin{enumerate}
\item Independent observations with distribution $N(\mu,1)$ are collected.

\item Interim analyses of the sum of the data are performed at $n$ and $2n$.

\item 
\begin{enumerate}
\item For the one-sided stopping rule, the trial is stopped either at the first $i n $, where $i \in \{1,2\}$, for which 
\begin{equation}
\sum_{k = 1}^{in} X_k \geq C (in)^\gamma,\label{eq:OneSidedStoppingRule}
\end{equation}
or at $3n$.

\item For the two-sided stopping rule, the trial is stopped either at the first $in$, where $i \in \{1,2\}$, for which
\begin{equation}
\left|\sum_{k = 1}^{in} X_k\right| \geq C (in)^\gamma,\label{eq:TwoSidedStoppingRule}
\end{equation}
or at $3n$.
\end{enumerate}
\end{enumerate}

Notice that the final sample size $N_n$ of the above trial can take the values $n$, $2n$, and $3n$. Furthermore, the above stopping rules are special cases of the general stopping rule (\ref{eq:StoppingRule}), where the one-sided stopping rule (\ref{eq:OneSidedStoppingRule}) corresponds to the choice
\begin{equation}
\psi_{C,1}(x) = \left\{\begin{array}{clrr}      
1 &\textrm{ if }& x \geq C\\       
0 &\textrm{ if }&  x< C
\end{array}\right.,\label{eq:PsiEx}
\end{equation}
and the two-sided stopping rule (\ref{eq:TwoSidedStoppingRule}) to the choice
\begin{equation}
\psi_{C,2}(x) = \left\{\begin{array}{clrr}      
1 &\textrm{ if }& \left|x\right| \geq C\\       
0 &\textrm{ if }&  \left|x\right| < C
\end{array}\right..\label{eq:PsiEx}
\end{equation}

For each fixed choice of the parameters $\mu$, $n$, $C$, and $\gamma$, we have calculated the following values for the sample mean 
$\widehat{\mu}_{N_n} = \frac{1}{N_n}\sum_{k = 1}^{N_n} X_k$ for the stopping rules (\ref{eq:OneSidedStoppingRule}) and (\ref{eq:TwoSidedStoppingRule}):

\begin{enumerate}
\item the {\em average lower  confidence limit}, the {\em average upper confidence limit}, and the {\em true coverage probability} for the interval $\left[\widehat{\mu}_{N_n} - 1.96/ \sqrt{N_n}, \widehat{\mu}_{N_n} + 1.96 /\sqrt{N_n}\right]$, which would be a $95 \%$ confidence interval for $\mu$ if $\sqrt{N_n} (\widehat{\mu}_{N_n} - \mu)$ were standard normally distributed, 
\item the {\em Kolmogorov distance} between the empirical distribution of $\sqrt{N_n} (\widehat{\mu}_{N_n} - \mu)$ and the standard normal distribution, which is close to $0$ if and only if $\sqrt{N_n} (\widehat{\mu}_{N_n} - \mu)$ is close to being standard normally distributed,
\item the {\em average length of the trial}.  
\end{enumerate}

From the calculated numbers, which can be found in Table 2 of the appendix, we can draw the following conclusions.

\begin{enumerate}
\item For $n \geq 50$, the true coverage probability of the interval 
$$\left[\widehat{\mu}_{N_n} - 1.96/ \sqrt{N_n}, \widehat{\mu}_{N_n} + 1.96 /\sqrt{N_n}\right]$$ 
is in all cases close to $95 \%$. So using this type of naive confidence interval after a group sequential trial turns out to be a reliable method.

\item Asympotic normality of the quantity $\sqrt{N_n} (\widehat{\mu}_{N_n} - \mu)$ is more subtle. In many cases, the Kolmogorov distance between the empirical distribution of  $\sqrt{N_n} (\widehat{\mu}_{N_n} - \mu)$ and the standard normal distribution is already below $0.05$ for $n = 50$, implying that the naively normalized sample mean is close to asymptotic normality. However, there are a few exceptions for which the Kolmogorov distance is substantially larger than $0.05$ for $n = 50$. These exceptions are listed in Table \ref{table:LargeD} below. If an exception turns out to be a case in which Theorem \ref{thm:AsymNorSeq} guarantees asymptotic normality, we have mentioned this.\ \\

\begin{center}
    \captionof{table}{Cases where the Kolmogorov distance is large for $n = 50$}
    \label{table:LargeD}
    \begin{tabular}{l l l l l l}
    \hline
    $\mu$ & C & $\gamma$ & one- or two-sided &Kol. Dist. &case in Theorem \ref{thm:AsymNorSeq} \\ \hline
     ~0&  2& 0 &  two-sided& 0.111& E (2)\\ \hline 
     ~0& 2&  0.25&two-sided &0.147& E (2)\\ \hline 
     ~0& 2&  0& one-sided&0.176& /\\\hline 
     ~0 &2&0.25&one-sided&0.130& /\\\hline 
     -1 &1&1&two-sided&0.190&/\\\hline 
     ~0 &1&0&two-sided&0.075&E (2)\\\hline 
     ~0 &1&0.25&two-sided&0.136&E (2)\\\hline 
     ~0 &1&0.5&two-sided&0.118&/\\\hline 
     ~1 &1&1&two-sided&0.187&/\\\hline 
     ~0 &1&0&one-sided&0.183&/\\\hline 
     ~0 &1&0.25&one-sided&0.158&/\\\hline 
     ~0 &1&0.5&one-sided&0.126&/\\\hline 
     ~1 &1&1&one-sided&0.187&/\\\hline 
     ~0 &0&arbitrary&one-sided&0.187&/\\\hline 

    \end{tabular}
\end{center}
\ \\
We conclude that the quantity $\sqrt{N_n} (\widehat{\mu}_{N_n} - \mu)$  is close to a standard normal distribution for $n = 50$ in all cases in which Theorem \ref{thm:AsymNorSeq}  guarantees asymptotic normality, except in the case E (2). Indeed, E (2) is the only case in which asymptotic normality is reached more slowly.

\end{enumerate}
\section{Concluding remarks}\label{sec:ConRem}

In this paper, we have studied a broad class of group sequential trials, with random length $N_n$, in which independent observations $X_1,X_2, \ldots$ with law $N(\mu,\sigma^2)$ are collected, and in which interim analyses of the sum of the observations are performed. After each interim analysis, one decides whether the trial is stopped or continued, based on the stopping rule (\ref{eq:StoppingRule}), with $\psi$ a fixed Borel measurable map of $\mathbb{R}$ into $\left[0,1\right]$, $\gamma \in \mathbb{R}^+_0$ a fixed shape parameter, and, for $m \in \mathbb{N}_0$, $K_m$ the sum of the first $m$ observations. We have illustrated that this mathematical model contains several interesting settings, which have been studied extensively in the literature. In particular, the choice $\gamma = 1/2$ leads to the Pocock boundaries, and the choice $\gamma = 0$ to the O'Brien-Fleming boundaries.

In the above described group sequential trial setting, we have studied the asymptotic distribution of the sample mean $\widehat{\mu}_{N_n} = K_{N_n}/N_n$, where, as $n \rightarrow \infty$, the maximal number of observations in the trial increases to infinity, while the number $L$ of interim analyses is kept fixed, and, for every $i \in \{1,\ldots,L\}$, the ratio of the number of observations collected at the $i$-th interim analysis to the maximal number of observations is also kept fixed.

We have shown in Theorem \ref{thm:AsymNorSeq}, our main result, that, in many realistic cases, the quantity $\frac{\sqrt{N_n}}{\sigma}\left(\widehat{\mu}_{N_n} - \mu\right)$ is asymptotically normal for the total variation distance, from which we have derived that it is often safe to rely on confidence intervals of the type $\left[\widehat{\mu}_{N_n} - \frac{\sigma}{\sqrt{N_n}} x , \widehat{\mu}_{N_n} + \frac{\sigma}{\sqrt{N_n}} x \right]$ for $\mu$. A simulation study has confirmed our theoretical findings.

\clearpage
\setcounter{page}{1}
\renewcommand{\thepage}{A.\arabic{page}}

\begin{center}

\title\bfseries{On asymptotic normality in estimation after a group sequential trial}

\vspace*{4mm}

\author{Ben Berckmoes, Anna Ivanova, Geert Molenberghs}

\vspace*{4mm}

{\bfseries\Large Appendix}\label{sec:Appx}

\end{center}

\begin{center}
\begin{longtable}{llllccccr}
\caption{\em Simulation results for the normal case with the number of simulated i.i.d.\ draws per sample 400, the number of simulated samples 1000. Values for standard deviation is kept fixed, $\sigma=1$. CL: confidence limit. Cov.Prob: coverage probability. Kol.Dist.: Kolmogorov distance.\label{simul1}}\\
\hline
\hline
  $\mu$&$n$&$C$&$\gamma$&Lower CL& Upper CL& Cov.Prob.&Kol.Dist.&Aver.Size\\
\hline
\hline
\endfirsthead
\multicolumn{9}{c}%
{\tablename\ \thetable\ -- \it{Continued from previous page}} \\
\hline
\hline
  $\mu$&$n$&$C$&$\gamma$&Lower CL& Upper CL& Cov.Prob.&Kol.Dist. &Aver.Size\\\hline
\hline
\endhead
\hline \multicolumn{9}{r}{\it{Continued on next page}} \\
\endfoot
\hline
\endlastfoot

\multicolumn{9}{c}{no stopping} \\
\hline
-1& & & &-1.098&-0.902&0.955&0.021&400\\
~0& & & &-0.098&~0.098&0.955&0.021&400\\
~1&	&	&	&~0.902&~1.098&0.955&0.021&400\\
\hline
\multicolumn{9}{c}{two-sided rule}\\
\hline
-1&10&2&0&-1.602&-0.403&0.914&0.019&10\\
-1&10&2&0.25&-1.603&-0.407&0.919&0.019&10\\
-1&10&2&0.50&-1.612&-0.453&0.926&0.056&11\\
-1&10&2&0.75&-1.547&-0.625&0.923&0.131&20\\
-1&10&2&1&-1.358&-0.647&0.939&0.018&30\\
-1&10&2&2&-1.357&-0.647&0.939&0.018&30\\
~0&10&2&0&-0.502&~0.503&0.906&0.126&17\\
~0&10&2&0.25&-0.435&~0.429&0.897&0.109&23\\
~0&10&2&0.50&-0.372&~0.360&0.900&0.036&29\\
~0&10&2&0.75&-0.357&~0.354&0.938&0.018&30\\
~0&10&2&1&-0.357&~0.353&0.939&0.018&30\\
~0&10&2&2&-0.357&~0.353&0.939&0.018&30\\
~1&10&2&0&~0.405&~1.604&0.915&0.019&10\\
~1&10&2&0.25&~0.414&~1.607&0.925&0.019&10\\
~1&10&2&0.50&~0.456&~1.618&0.942&0.064&11\\
~1&10&2&0.75&~0.628&~1.552&0.926&0.139&20\\
~1&10&2&1&~0.644&~1.355&0.937&0.018&30\\
~1&10&2&2&~0.643&~1.353&0.939&0.018&30\\
\hline
\multicolumn{9}{c}{one-sided rule}\\
\hline
-1&10&2&0&-1.357&-0.647&0.939&0.018&30\\
-1&10&2&0.25&-1.357&-0.647&0.939&0.018&30\\
-1&10&2&0.50&-1.357&-0.647&0.939&0.018&30\\
-1&10&2&0.75&-1.357&-0.647&0.939&0.018&30\\
-1&10&2&1&-1.357&-0.647&0.939&0.018&30\\
-1&10&2&2&-1.357&-0.647&0.939&0.018&30\\
~0&10&2&0&-0.351&~0.509&0.924&0.149&23\\
~0&10&2&0.25&-0.340&~0.445&0.919&0.099&26\\
~0&10&2&0.50&-0.350&~0.369&0.922&0.026&29\\
~0&10&2&0.75&-0.357&~0.354&0.938&0.018&30\\
~0&10&2&1&-0.357&~0.353&0.939&0.018&30\\
~0&10&2&2&-0.357&~0.353&0.939&0.018&30\\
~1&10&2&0&~0.405&~1.604&0.915&0.019&10\\
~1&10&2&0.25&~0.414&~1.607&0.925&0.019&10\\
~1&10&2&0.50&~0.456&~1.618&0.942&0.064&11\\
~1&10&2&0.75&~0.628&~1.552&0.926&0.139&20\\
~1&10&2&1&~0.644&~1.355&0.937&0.018&30\\
~1&10&2&2&~0.643&~1.353&0.939&0.018&30\\
\hline
\multicolumn{9}{c}{two-sided rule}\\
\hline
-1&50&2&0&-1.279&-0.727&0.939&0.023&50\\
-1&50&2&0.25&-1.279&-0.727&0.939&0.023&50\\
-1&50&2&0.50&-1.279&-0.727&0.939&0.023&50\\
-1&50&2&0.75&-1.282&-0.736&0.958&0.036&52\\
-1&50&2&1&-1.159&-0.839&0.957&0.020&150\\
-1&50&2&2&-1.159&-0.839&0.957&0.020&150\\
~0&50&2&0&-0.262&~0.255&0.938&0.111&62\\
~0&50&2&0.25&-0.226&~0.217&0.928&0.147&93\\
~0&50&2&0.50&-0.165&~0.166&0.913&0.020&145\\
~0&50&2&0.75&-0.159&~0.161&0.957&0.020&150\\
~0&50&2&1&-0.159&~0.161&0.957&0.020&150\\
~0&50&2&2&-0.159&~0.161&0.957&0.020&150\\
~1&50&2&0&~0.721&~1.273&0.939&0.023&50\\
~1&50&2&0.25&~0.721&~1.273&0.939&0.023&50\\
~1&50&2&0.50&~0.721&~1.273&0.939&0.023&50\\
~1&50&2&0.75&~0.731&~1.276&0.955&0.026&52\\
~1&50&2&1&~0.841&~1.161&0.957&0.020&150\\
~1&50&2&2&~0.841&~1.161&0.957&0.020&150\\
\hline
\multicolumn{9}{c}{one-sided rule}\\
\hline
-1&50&2&0&-1.159&-0.839&0.957&0.020&150\\
-1&50&2&0.25&-1.159&-0.839&0.957&0.020&150\\
-1&50&2&0.50&-1.159&-0.839&0.957&0.020&150\\
-1&50&2&0.75&-1.159&-0.839&0.957&0.020&150\\
-1&50&2&1&-1.159&-0.839&0.957&0.020&150\\
-1&50&2&2&-1.159&-0.839&0.957&0.020&150\\
~0&50&2&0&-0.172&~0.249&0.947&0.176&105\\
~0&50&2&0.25&-0.158&~0.221&0.943&0.130&122\\
~0&50&2&0.50&-0.157&~0.169&0.936&0.020&147\\
~0&50&2&0.75&-0.159&~0.161&0.957&0.020&150\\
~0&50&2&1&-0.159&~0.161&0.957&0.020&150\\
~0&50&2&2&-0.159&~0.161&0.957&0.020&150\\
~1&50&2&0&~0.721&~1.273&0.939&0.023&50\\
~1&50&2&0.25&~0.721&~1.273&0.939&0.023&50\\
~1&50&2&0.50&~0.721&~1.273&0.939&0.023&50\\
~1&50&2&0.75&~0.731&~1.276&0.955&0.026&52\\
~1&50&2&1&~0.841&~1.161&0.957&0.020&150\\
~1&50&2&2&~0.841&~1.161&0.957&0.020&150\\
\hline
\multicolumn{9}{c}{two-sided rule}\\
\hline
-1&100&2&0&-1.196&-0.805&0.955&0.018&100\\
-1&100&2&0.25&-1.196&-0.805&0.955&0.018&100\\
-1&100&2&0.50&-1.196&-0.805&0.955&0.018&100\\
-1&100&2&0.75&-1.196&-0.805&0.955&0.018&100\\
-1&100&2&1&-1.113&-0.887&0.954&0.020&300\\
-1&100&2&2&-1.113&-0.887&0.954&0.020&300\\
~0&100&2&0&-0.186&~0.185&0.955&0.073&120\\
~0&100&2&0.25&-0.163&~0.160&0.948&0.142&175\\
~0&100&2&0.50&-0.117&~0.116&0.919&0.019&291\\
~0&100&2&0.75&-0.113&~0.113&0.954&0.020&300\\
~0&100&2&1&-0.113&~0.113&0.954&0.020&300\\
~0&100&2&2&-0.113&~0.113&0.954&0.020&300\\
~1&100&2&0&~0.804&~1.195&0.955&0.018&100\\
~1&100&2&0.25&~0.804&~1.195&0.955&0.018&100\\
~1&100&2&0.50&~0.804&~1.195&0.955&0.018&100\\
~1&100&2&0.75&~0.804&~1.195&0.955&0.018&100\\
~1&100&2&1&~0.887&~1.113&0.954&0.020&300\\
~1&100&2&2&~0.887&~1.113&0.954&0.020&300\\
\hline
\multicolumn{9}{c}{one-sided rule}\\
\hline
-1&100&2&0&-1.113&-0.887&0.954&0.020&300\\
-1&100&2&0.25&-1.113&-0.887&0.954&0.012&300\\
-1&100&2&0.50&-1.113&-0.887&0.954&0.020&300\\
-1&100&2&0.75&-1.113&-0.887&0.954&0.020&300\\
-1&100&2&1&-1.113&-0.887&0.954&0.020&300\\
-1&100&2&2&-1.113&-0.887&0.954&0.020&300\\
~0&100&2&0&-0.123&~0.178&0.961&0.183&204\\
~0&100&2&0.25&-0.114&~0.160&0.957&0.147&237\\
~0&100&2&0.50&-0.111&~0.118&0.940&0.019&296\\
~0&100&2&0.75&-0.113&~0.113&0.954&0.020&300\\
~0&100&2&1&-0.113&~0.113&0.954&0.020&300\\
~0&100&2&2&-0.113&~0.113&0.954&0.020&300\\
~1&100&2&0&~0.804&~1.195&0.955&0.018&100\\
~1&100&2&0.25&~0.804&~1.195&0.955&0.018&100\\
~1&100&2&0.50&~0.804&~1.195&0.955&0.018&100\\
~1&100&2&0.75&~0.804&~1.195&0.955&0.018&100\\
~1&100&2&1&~0.887&~1.113&0.954&0.020&300\\
~1&100&2&2&~0.887&~1.113&0.954&0.020&300\\
\hline
\multicolumn{9}{c}{two-sided rule}\\
\hline
-1&500&2&0&-1.087&-0.912&0.957&0.020&500\\
-1&500&2&0.25&-1.087&-0.912&0.957&0.020&500\\
-1&500&2&0.50&-1.087&-0.912&0.957&0.020&500\\
-1&500&2&0.75&-1.087&-0.912&0.957&0.020&500\\
-1&500&2&1&-1.050&-0.949&0.953&0.024&1500\\
-1&500&2&2&-1.050&-0.949&0.953&0.024&1500\\
~0&500&2&0&-0.085&~0.086&0.957&0.050&540\\
~0&500&2&0.25&-0.077&~0.079&0.953&0.134&719\\
~0&500&2&0.50&-0.052&~0.053&0.901&0.030&1440\\
~0&500&2&0.75&-0.050&~0.051&0.953&0.024&1500\\
~0&500&2&1&-0.050&~0.051&0.953&0.024&1500\\
~0&500&2&2&-0.050&~0.051&0.953&0.024&1500\\
~1&500&2&0&~0.913&~1.088&0.957&0.020&500\\
~1&500&2&0.25&~0.913&~1.088&0.957&0.020&500\\
~1&500&2&0.50&~0.913&~1.088&0.957&0.020&500\\
~1&500&2&0.75&~0.913&~1.088&0.957&0.020&500\\
~1&500&2&1&~0.950&~1.051&0.953&0.024&1500\\
~1&500&2&2&~0.950&~1.051&0.953&0.024&1500\\
\hline
\multicolumn{9}{c}{one-sided rule}\\
\hline
-1&500&2&0&-1.050&-0.949&0.953&0.024&1500\\
-1&500&2&0.25&-1.050&-0.949&0.953&0.024&1500\\
-1&500&2&0.50&-1.050&-0.949&0.953&0.024&1500\\
-1&500&2&0.75&-1.050&-0.949&0.953&0.024&1500\\
-1&500&2&1&-1.050&-0.949&0.953&0.024&1500\\
-1&500&2&2&-1.050&-0.949&0.953&0.024&1500\\
~0&500&2&0&-0.057&~0.082&0.963&0.194&963\\
~0&500&2&0.25&-0.052&~0.078&0.961&0.201&1083\\
~0&500&2&0.50&-0.049&~0.054&0.927&0.030&1469\\
~0&500&2&0.75&-0.050&~0.051&0.953&0.024&1500\\
~0&500&2&1&-0.050&~0.051&0.953&0.024&1500\\
~0&500&2&2&-0.050&~0.051&0.953&0.024&1500\\
~1&500&2&0&~0.913&~1.088&0.957&0.020&500\\
~1&500&2&0.25&~0.913&~1.088&0.957&0.020&500\\
~1&500&2&0.50&~0.913&~1.088&0.957&0.020&500\\
~1&500&2&0.75&~0.913&~1.088&0.957&0.020&500\\
~1&500&2&1&~0.950&~1.051&0.953&0.024&1500\\
~1&500&2&2&~0.950&~1.051&0.953&0.024&1500\\
\hline
\multicolumn{9}{c}{two-sided rule}\\
\hline
-1&10&1&0&-1.600&-0.401&0.912&0.019&10\\
-1&10&1&0.25&-1.601&-0.402&0.913&0.019&10\\
-1&10&1&0.50&-1.603&-0.406&0.918&0.019&10\\
-1&10&1&0.75&-1.610&-0.439&0.929&0.041&11\\
-1&10&1&1&-1.575&-0.601&0.926&0.184&19\\
-1&10&1&2&-1.357&-0.647&0.939&0.018&30\\
~0&10&1&0&-0.553&~0.556&0.910&0.106&13\\
~0&10&1&0.25&-0.509&~0.515&0.907&0.142&16\\
~0&10&1&0.50&-0.448&~0.441&0.897&0.128&22\\
~0&10&1&0.75&-0.379&~0.369&0.887&0.042&28\\
~0&10&1&1&-0.357&~0.355&0.937&0.018&30\\
~0&10&1&2&-0.357&~0.353&0.939&0.018&30\\
~1&10&1&0&~0.403&~1.602&0.912&0.019&10\\
~1&10&1&0.25&~0.405&~1.604&0.914&0.019&10\\
~1&10&1&0.50&~0.410&~1.606&0.919&0.019&10\\
~1&10&1&0.75&~0.443&~1.617&0.940&0.044&11\\
~1&10&1&1&~0.602&~1.576&0.928&0.176&19\\
~1&10&1&2&~0.643&~1.353&0.939&0.018&30\\
\hline
\multicolumn{9}{c}{one-sided rule}\\
\hline
-1&10&1&0&-1.357&-0.646&0.938&0.018&30\\
-1&10&1&0.25&-1.357&-0.646&0.938&0.018&30\\
-1&10&1&0.50&-1.357&-0.647&0.939&0.018&30\\
-1&10&1&0.75&-1.357&-0.647&0.939&0.018&30\\
-1&10&1&1&-1.357&-0.647&0.939&0.018&30\\
-1&10&1&2&-1.357&-0.647&0.939&0.018&30\\
~0&10&1&0&-0.370&0.547&0.927&0.178&21\\
~0&10&1&0.25&-0.353&0.518&0.924&0.165&23\\
~0&10&1&0.50&-0.340&0.456&0.919&0.115&26\\
~0&10&1&0.75&-0.347&0.3812&0.916&0.031&29\\
~0&10&1&1&-0.356&0.355&0.937&0.018&30\\
~0&10&1&2&-0.3571&0.353&0.939&0.018&30\\
~1&10&1&0&~0.403&1.602&0.912&0.019&30\\
~1&10&1&0.25&~0.405&1.604&0.914&0.019&10\\
~1&10&1&0.50&~0.410&1.606&0.919&0.019&10\\
~1&10&1&0.75&~0.443&1.617&0.940&0.044&11\\
~1&10&1&1&~0.602&1.576&0.928&0.176&19\\
~1&10&1&2&~0.643&1.353&0.939&0.018&30\\
\hline
\multicolumn{9}{c}{two-sided rule}\\
\hline
-1&50&1&0&-1.279&-0.727&0.939&0.023&50\\
-1&50&1&0.25&-1.279&-0.727&0.939&0.023&50\\
-1&50&1&0.50&-1.279&-0.727&0.939&0.023&50\\
-1&50&1&0.75&-1.279&-0.727&0.939&0.023&50\\
-1&50&1&1&-1.266&-0.819&0.949&0.190&93\\
-1&50&1&2&-1.159&-0.839&0.957&0.020&150\\
~0&50&1&0&-0.272&~0.263&0.938&0.075&56\\
~0&50&1&0.25&-0.255&~0.247&0.935&0.135&69\\
~0&50&1&0.50&-0.205&~0.201&0.926&0.126&110\\
~0&50&1&0.75&-0.161&~0.161&0.948&0.020&149\\
~0&50&1&1&-0.159&~0.161&0.957&0.020&150\\
~0&50&1&2&-0.159&~0.161&0.957&0.020&150\\
~1&50&1&0&~0.721&~1.273&0.939&0.023&50\\
~1&50&1&0.25&~0.721&~1.273&0.939&0.023&50\\
~1&50&1&0.50&~0.721&~1.273&0.939&0.023&50\\
~1&50&1&0.75&~0.721&~1.273&0.939&0.023&50\\
~1&50&1&1&~0.819&~1.261&0.948&0.187&95\\
~1&50&1&2&~0.841&~1.161&0.957&0.020&150\\
\hline
\multicolumn{9}{c}{one-sided rule}\\
\hline
-1&50&1&0&-1.159&-0.839&0.957&0.020&150\\
-1&50&1&0.25&-1.159&-0.839&0.957&0.020&150\\
-1&50&1&0.50&-1.159&-0.839&0.957&0.020&150\\
-1&50&1&0.75&-1.159&-0.839&0.957&0.020&150\\
-1&50&1&1&-1.159&-0.839&0.957&0.020&150\\
-1&50&1&2&-1.159&-0.839&0.957&0.020&150\\
~0&50&1&0&-0.177&~0.255&0.947&0.183&100\\
~0&50&1&0.25&-0.168&~0.243&0.946&0.158&109\\
~0&50&1&0.50&-0.155&~0.208&0.942&0.126&130\\
~0&50&1&0.75&-0.158&~0.162&0.953&0.020&150\\
~0&50&1&1&-0.159&~0.161&0.957&0.020&150\\
~0&50&1&2&-0.159&~0.161&0.957&0.020&150\\
~1&50&1&0&0.721&~1.273&0.939&0.023&50\\
~1&50&1&0.25&~0.721&~1.273&0.939&0.023&50\\
~1&50&1&0.50&~0.721&~1.273&0.939&0.023&50\\
~1&50&1&0.75&~0.721&~1.273&0.939&0.023&50\\
~1&50&1&1&~0.819&~1.261&0.948&0.187&95\\
~1&50&1&2&~0.841&~1.161&0.957&0.020&150\\
\hline
\multicolumn{9}{c}{two-sided rule}\\
\hline
-1&100&1&0&-1.196&-0.805&0.955&0.018&100\\
-1&100&1&0.25&-1.196&-0.805&0.955&0.018&100\\
-1&100&1&0.50&-1.196&-0.805&0.955&0.018&100\\
-1&100&1&0.75&-1.196&-0.805&0.955&0.018&100\\
-1&100&1&1&-1.185&-0.869&0.949&0.184&187\\
-1&100&1&2&-1.113&-0.887&0.954&0.020&300\\
~0&100&1&0&-0.191&~0.190&0.955&0.041&109\\
~0&100&1&0.25&-0.179&~0.178&0.955&0.111&135\\
~0&100&1&0.50&-0.141&~0.143&0.939&0.118&225\\
~0&100&1&0.75&-0.113&~0.113&0.954&0.020&300\\
~0&100&1&1&-0.113&~0.113&0.954&0.020&300\\
~0&100&1&2&-0.113&~0.113&0.954&0.020&300\\
~1&100&1&0&~0.804&~1.195&0.955&0.018&100\\
~1&100&1&0.25&~0.804&~1.195&0.955&0.018&100\\
~1&100&1&0.50&~0.804&~1.195&0.955&0.018&100\\
~1&100&1&0.75&~0.804&~1.195&0.955&0.018&100\\
~1&100&1&1&~0.871&~1.184&0.963&0.188&189\\
~1&100&1&2&~0.887&~1.113&0.954&0.020&300\\
\hline
\multicolumn{9}{c}{one-sided rule}\\
\hline
-1&100&1&0&-1.113&-0.887&0.954&0.020&300\\
-1&100&1&0.25&-1.113&-0.887&0.954&0.020&300\\
-1&100&1&0.50&-1.113&-0.887&0.954&0.020&300\\
-1&100&1&0.75&-1.113&-0.887&0.954&0.020&300\\
-1&100&1&1&-1.113&-0.887&0.954&0.020&300\\
-1&100&1&2&-1.113&-0.887&0.954&0.020&300\\
~0&100&1&0&-0.126&~0.181&0.961&0.184&196\\
~0&100&1&0.25&-0.120&~0.173&0.960&0.178&213\\
~0&100&1&0.50&-0.109&~0.145&0.951&0.109&263\\
~0&100&1&0.75&-0.113&~0.114&0.954&0.020&300\\
~0&100&1&1&-0.113&~0.113&0.954&0.020&300\\
~0&100&1&2&-0.113&~0.113&0.954&0.020&300\\
~1&100&1&0&~0.804&~1.195&0.955&0.018&100\\
~1&100&1&0.25&~0.804&~1.195&0.955&0.018&100\\
~1&100&1&0.50&~0.804&~1.195&0.955&0.018&100\\
~1&100&1&0.75&~0.804&~1.195&0.955&0.018&100\\
~1&100&1&1&~0.871&~1.184&0.963&0.188&189\\
~1&100&1&2&~0.887&~1.113&0.954&0.020&300\\
\hline
\multicolumn{9}{c}{two-sided rule}\\
\hline
-1&500&1&0&-1.087&-0.912&0.957&0.020&500\\
-1&500&1&0.25&-1.087&-0.912&0.957&0.020&500\\
-1&500&1&0.50&-1.087&-0.912&0.957&0.020&500\\
-1&500&1&0.75&-1.087&-0.912&0.957&0.020&500\\
-1&500&1&1&-1.082&-0.941&0.949&0.187&945\\
-1&500&1&2&-1.050&-0.949&0.953&0.024&1500\\
~0&500&1&0&-0.086&~0.087&0.957&0.035&519\\
~0&500&1&0.25&-0.082&~0.084&0.956&0.090&596\\
~0&500&1&0.50&-0.063&~0.064&0.940&0.126&1120\\
~0&500&1&0.75&-0.050&~0.051&0.953&0.024&1500\\
~0&500&1&1&-0.050&~0.051&0.953&0.024&1500\\
~0&500&1&2&-0.050&~0.051&0.953&0.024&1500\\
~1&500&1&0&~0.913&~1.088&0.957&0.020&500\\
~1&500&1&0.25&~0.913&~1.088&0.957&0.020&500\\
~1&500&1&0.50&~0.913&~1.088&0.957&0.020&500\\
~1&500&1&0.75&~0.913&~1.088&0.957&0.020&500\\
~1&500&1&1&~0.942&~1.083&0.964&0.186&936\\
~1&500&1&2&~0.950&~1.051&0.953&0.024&1500\\
\hline
\multicolumn{9}{c}{one-sided rule}\\
\hline
-1&500&1&0&-1.050&-0.949&0.953&0.024&1500\\
-1&500&1&0.25&-1.050&-0.949&0.953&0.024&1500\\
-1&500&1&0.50&-1.050&-0.949&0.953&0.024&1500\\
-1&500&1&0.75&-1.0503&-0.949&0.953&0.024&1500\\
-1&500&1&1&-1.050&-0.949&0.953&0.024&1500\\
-1&500&1&2&-1.050&-0.949&0.953&0.024&1500\\
~0&500&1&0&-0.057&~0.083&0.963&0.189&949\\
~0&500&1&0.25&-0.055&~0.081&0.962&0.207&1004\\
~0&500&1&0.50&-0.049&~0.066&0.953&0.128&1305\\
~0&500&1&0.75&-0.050&~0.051&0.953&0.024&1500\\
~0&500&1&1&-0.050&~0.051&0.953&0.024&1500\\
~0&500&1&2&-0.050&~0.051&0.953&0.024&1500\\
~1&500&1&0&~0.913&~1.088&0.957&0.020&500\\
~1&500&1&0.25&~0.913&~1.088&0.957&0.020&500\\
~1&500&1&0.50&~0.913&~1.088&0.957&0.020&500\\
~1&500&1&0.75&~0.913&~1.088&0.957&0.020&500\\
~1&500&1&1&~0.942&~1.083&0.964&0.186&936\\
~1&500&1&2&~0.950&~1.051&0.953&0.024&1500\\
\hline
\multicolumn{9}{c}{one-sided rule}\\
\hline
-1&10&0&0&-1.356&-0.645&0.937&0.018&30\\
-1&10&0&0.25&-1.356&-0.645&0.937&0.018&30\\
-1&10&0&0.50&-1.356&-0.645&0.937&0.018&30\\
-1&10&0&0.75&-1.356&-0.645&0.937&0.018&30\\
-1&10&0&1&-1.356&-0.645&0.937&0.018&30\\
-1&10&0&2&-1.356&-0.645&0.937&0.018&30\\
~0&10&0&0&-0.398&~0.576&0.928&0.176&19\\
~0&10&0&0.25&-0.398&~0.576&0.928&0.176&19\\
~0&10&0&0.50&-0.398&~0.576&0.928&0.176&19\\
~0&10&0&0.75&-0.398&~0.576&0.928&0.176&19\\
~0&10&0&1&-0.398&~0.576&0.928&0.176&19\\
~0&10&0&2&-0.398&~0.576&0.928&0.176&19\\
~1&10&0&0&~0.402&~1.602&0.911&0.019&10\\
~1&10&0&0.25&~0.402&~1.602&0.911&0.019&10\\
~1&10&0&0.50&~0.402&~1.602&0.911&0.019&10\\
~1&10&0&0.75&~0.402&~1.602&0.911&0.019&10\\
~1&10&0&1&~0.402&~1.602&0.911&0.019&10\\
~1&10&0&2&~0.402&~1.602&0.911&0.019&10\\
\hline
\multicolumn{9}{c}{one-sided rule}\\
\hline
-1&50&0&0&-1.159&-0.839&0.957&0.020&150\\
-1&50&0&0.25&-1.159&-0.839&0.957&0.020&150\\
-1&50&0&0.50&-1.159&-0.839&0.957&0.020&150\\
-1&50&0&0.75&-1.159&-0.839&0.957&0.020&150\\
-1&50&0&1&-1.159&-0.839&0.957&0.020&150\\
-1&50&0&2&-1.159&-0.839&0.957&0.020&150\\
~0&50&0&0&-0.181&0.261&0.948&0.187&95\\ 
~0&50&0&0.25&-0.181&0.261&0.948&0.187&95\\
~0&50&0&0.50&-0.181&0.261&0.948&0.187&95\\
~0&50&0&0.75&-0.181&0.261&0.948&0.187&95\\
~0&50&0&1&-0.181&0.261&0.948&0.187&95\\
~0&50&0&2&-0.181&0.261&0.948&0.187&95\\
~1&50&0&0&~0.721&1.273&0.939&0.023&50\\
~1&50&0&0.25&~0.721&1.273&0.939&0.023&50\\
~1&50&0&0.50&~0.721&1.273&0.939&0.023&50\\
~1&50&0&0.75&~0.721&1.273&0.939&0.023&50\\
~1&50&0&1&~0.721&1.273&0.939&0.023&50\\
~1&50&0&2&~0.721&1.273&0.939&0.023&50\\
\hline
\multicolumn{9}{c}{one-sided rule}\\
\hline
-1&100&0&0&-1.113&-0.887&0.954&0.020&300\\
-1&100&0&0.25&-1.113&-0.887&0.954&0.020&300\\
-1&100&0&0.50&-1.113&-0.887&0.954&0.020&300\\
-1&100&0&0.75&-1.113&-0.887&0.954&0.020&300\\
-1&100&0&1&-1.113&-0.88&0.954&0.020&300\\
-1&100&0&2&-1.113&-0.887&0.954&0.020&300\\
~0&100&0&0&-0.129&0.184&0.963&0.188&189\\
~0&100&0&0.25&-0.129&0.184&0.963&0.188&189\\
~0&100&0&0.50&-0.129&0.184&0.963&0.188&189\\
~0&100&0&0.75&-0.129&0.184&0.963&0.188&189\\
~0&100&0&1&-0.129&0.184&0.963&0.188&189\\
~0&100&0&2&-0.129&0.184&0.963&0.188&189\\
~1&100&0&0&~0.804&1.195&0.955&0.011&100\\
~1&100&0&0.25&~0.804&1.195&0.955&0.011&100\\
~1&100&0&0.50&~0.804&1.195&0.955&0.018&100\\
~1&100&0&0.75&~0.804&1.195&0.955&0.018&100\\
~1&100&0&1&~0.804&1.195&0.955&0.018&100\\
~1&100&0&2&~0.804&1.195&0.955&0.018&100\\
\hline
\multicolumn{9}{c}{one-sided rule}\\
\hline
-1&500&0&0&-1.050&-0.949&0.953&0.024&1500\\
-1&500&0&0.25&-1.050&-0.949&0.953&0.024&1500\\
-1&500&0&0.50&-1.050&-0.949&0.953&0.024&1500\\
-1&500&0&0.75&-1.050&-0.949&0.953&0.024&1500\\
-1&500&0&1&-1.050&-0.949&0.953&0.024&1500\\
-1&500&0&2&-1.050&-0.949&0.953&0.024&1500\\
~0&500&0&0&-0.058&~0.083&0.964&0.186&936\\
~0&500&0&0.25&-0.058&~0.083&0.964&0.186&936\\
~0&500&0&0.50&-0.058&~0.083&0.964&0.186&936\\
~0&500&0&0.75&-0.058&~0.083&0.964&0.186&936\\
~0&500&0&1&-0.058&~0.083&0.964&0.186&936\\
~0&500&0&2&-0.058&~0.083&0.964&0.186&936\\
~1&500&0&0&~0.913&~1.088&0.957&0.020&500\\
~1&500&0&0.25&~0.913&~1.088&0.957&0.020&500\\
~1&500&0&0.50&~0.913&~1.088&0.957&0.020&500\\
~1&500&0&0.75&~0.913&~1.088&0.957&0.020&500\\
~1&500&0&1&~0.913&~1.088&0.957&0.020&500\\
~1&500&0&2&~0.913&~1.088&0.957&0.020&500\\
\hline
\multicolumn{9}{c}{two-sided rule}\\
\hline
~0&100&1&0.25&-0.179&~0.178&0.955&0.111&135\\
~0&500&1&0.25&-0.082&~0.084&0.956&0.090&596\\
~0&1000&1&0.25&-0.059&~0.061&0.942&0.086&1137\\
~0&2000&1&0.25&-0.041&~0.043&0.961&0.066&2312\\
\hline
\multicolumn{9}{c}{two-sided rule}\\
\hline
~0&100&2&0.25&-0.163&~0.160&0.948&0.142&175\\
~0&500&2&0.25&-0.077&~0.079&0.953&0.134&719\\
~0&1000&2&0.25&-0.056&~0.058&0.940&0.119&1331\\
~0&2000&2&0.25&-0.039&~0.041&0.961&0.101&2666\\

\hline
\hline
\end{longtable}
\end{center}

\end{document}